\documentclass[review]{elsarticle}

\usepackage{lineno,hyperref}
\usepackage[utf8]{inputenc}
\usepackage[english]{babel}
\usepackage{amssymb}
\usepackage{ifthen}
\modulolinenumbers[5]

\newtheorem{theorem}{Theorem}
\newtheorem{pro}{Proposition}
\newtheorem{corol}{Corollary}
\newproof{proof}{Proof}

\def\proofend{$\blacktriangle$\vspace{0.3em}\par}

\journal{Discrete Applied Mathematics}

\begin{document}

\begin{frontmatter}

\title{On  the number of transversals in latin squares}

\author{Vladimir N. Potapov\fnref{fn1}}
\ead{vpotapov@math.nsc.ru}

 \fntext[fn1]{The work was funded by the
Russian Science Foundation (grant No 14-11-00555).  }

\address{Sobolev Institute of Mathematics, 4 Acad. Koptyug Avenue, Novosibirsk, Russia}

\begin{abstract}
The logarithm of the maximum number of transversals over all latin
squares of order $n$ is greater than $\frac{n}{6}(\ln n+ O(1))$.
\end{abstract}


\end{frontmatter}


\section{Introduction}

A {\it latin square} of order $n$ is an $n\times n$ array of $n$
symbols  in which each symbol occurs exactly once in each row and in
each column. This property ensures that a latin square is the Cayley
table of a finite quasigroup. It is often convenient to represent a
latin square as a graph of the corresponding quasigroup, i.\,e. as a
set of ordered triples. Without loss of generality  suppose that the
set of symbols of the latin square is $\{0,1,\dots,n-1\}$. A latin
square $A$ of order $n$ is said to contain a {\it proper latin
subsquare} of order $m$, $1 < m < n$, if there exists an
intersection of $m$ rows and  $m$  columns within $A$ that is also a
latin square.

 A diagonal of a square is a set of entries that
contains exactly one representative of each row and column. A {\it
transversal} is a diagonal in which no symbol is repeated. A pair of
latin squares $A = (a_{ij})$ and $B = (b_{ij})$ of order $n$ are
said to be {\it orthogonal} mates if the $n^2 $ ordered pairs
$(a_{ij} , b_{ij})$ are distinct. Thus, if we look at all $n$
occurrences of a given symbol in $B$, the corresponding positions in
$A$ must form a transversal.

Denote by $t(A)$ the number of transversals of a latin square $A$.
Define $T(n)=\max t(A)$ to be  the maximum number of transversals
over all the latin squares of order $n$.  The best upper bound of
$T(n)$ is  due to Taranenko \cite{TA}

\begin{equation}\label{tAT}  T(n)\leq n^ne^{-2n+o(n)}.
\end{equation}

 Let $B_n$
be the Cayley table of the cyclic group of order $n$. Vardi supposed
that there exist two real constants $c_1$ and $c_2$ such that
$c_1^nn! \leq t(B_n) \leq c_2^nn!$, where $0 < c_1 < c_2 < 1$ and
$n> 3$ is odd, but the known lower bound of $T(n)$ is only
exponential (see \cite{W07}, \cite{W11}).   Cavenagh and Wanless
\cite{CW10} proved that if $n$ is a sufficiently large odd integer
then $t(B_n)> (3.246)^n$. By MacNeish's theorem \cite{HCD}, the
following is true.
\begin{pro}
If $n=p_1^{i_1}\cdots p_k^{i_k}$, where numbers $p_i$ are prime, and
$m=\min p_j^{i_j}$, then there exists a set of $m-1$ mutually
orthogonal latin squares which include $B_n$.
\end{pro}

\section{Transversals in Steiner latin squares}

A set of $3$-element subsets (triples) of $n$-element set is called
{\it Steiner triple system} (STS) if each pair of elements is
contained in exactly one triple. A STS consists of $n(n-1)/6$
triples. A well-known necessary and sufficient condition for the
existence of STS is that $n\equiv 1\ {\rm or}\ 3\ {\rm mod}\ 6$.

As mentioned above, a latin square can be represented as the graph
of a quasigroup, i.\,e. as a set of ordered triples of the
$n$-element set such that each pair of elements occurs in each pair
of positions and a pair of elements of any triple defines the third
element of the triple. The first and the second elements of triples
define row and column, and the third element defines the symbol in
the corresponding entry of a latin square. Thus,  given a STS,  we
can obtain a latin square by replacing each unordered triple with
the six ordered triples and by adding $n$ triples of the form
$(a,a,a)$. This latin square is called Steiner (it is a table of a
Steiner quasigroup). By the inclusion-exclusion principle and the
definition of STS, it is easy to prove the following proposition.

\begin{pro}
 For any STS of order $n$  the union of $p$ disjoint triples of the STS
  intersects with  at most $ s(p)=3p(\frac{n-1}{2}-\frac{3p-1}{3})$ triples of the STS.
\end{pro}
\begin{proof}
Let $V$ be the union of  $p$ disjoint triples. For each point $v\in
V$ there are $(n-1)/2$ triples of the STS that include $v$.  Each
pair $v_1,v_2\in V$ is contained in some triple of the STS.
Consequently, $\frac{3p(3p-1)}{2}$ triples of the STS occur again.
The number of triples $v_1,v_2,v_3\in V$ included in the STS isn't
more than $\frac{3p(3p-1)}{6}$. So, by the inclusion-exclusion
principle $V$ intersects with at most
$3p(\frac{n-1}{2}-\frac{3p-1}{2}+\frac{3p-1}{6})$ triples of the
STS. \proofend
\end{proof}

\begin{theorem}\label{t01transv}
If $S_n$ is a Steiner latin square  then $t(S_n)\geq \frac{6^{\lceil
(n-1)/6\rceil-1}\lfloor\frac{n}{3}\rfloor!}{\lceil\frac{n-1}{6}\rceil!\lceil\frac{n-1}{6}\rceil}$.
\end{theorem}
\begin{proof}
Consider the STS corresponding to $S_n$. By definition, for each
triple $(a,b,c)$ of the STS there is the latin subsquare $\left(
 \begin{array}{ccc}
a & c & b \\
c & b & a \\
b& a &  c\\
  \end{array} \right)$ in the intersection of rows and columns labeled
  by $a, b, c$. This subsquare has three transversals.
  We will construct transversal $T$ of $S_n$ recursively. In the each step we will take
  a triple of STS  that is disjoint with triples taken before and we will add three elements corresponding to
  this triple to $T$.  Let $K(p)$ be the set
  of elements of triples taken after $p$ steps, $K(0)=\varnothing$ and $|K(p)|=3p$.
   By Proposition 2, elements of $K(p)$
  belong to $s(p)=3p(\frac{n-1}{2}-\frac{3p-1}{3})$ triples at most.
Thus at the $(p+1)$th step, it is possible to take one of $n(n-1)/6
-s(p)= (n-3p)(n-6p-1)/6$ triples that don't intersect with $K(p)$.
If we take the triple $(a,b,c)$ then we get $K(p+1)=K(p)\cup
\{a,b,c\}$ and add $3$
 entries (ordered triples) $\{(a,c,b),  (b,a,c), (c,b,a)\}$ or
  $\{ (a,b,c),(b,c,a), (c,a,b) \}$ to
  $T$.
If $n(n-1)/6 \leq s(p)$ i.e. $p\geq p_0=\lceil(n-1)/6\rceil$ then we
add to $T$ entries $(e,e,e)$ for all $e\not\in K(p)$. Consequently,
there exist more than
$\frac{1}{p_0!}\prod\limits_{p=0}^{p_0-1}2\lceil(n-3p)(n-6p-1)/6\rceil\geq
6^{p_0-1}\lfloor\frac{n}{3}\rfloor!(p_0-1)!/((\lfloor\frac{n}{3}\rfloor-p_0)!p_0!)$
variants to choose a transversal. \proofend
\end{proof}

In the last part of the note the proposed construction  of
transversals is adapted to latin squares of any large order.

Bose (see \cite{HCD}) proposed the following construction of STS.
Let $A$ be a latin square of order $n$ corresponding to idempotent,
commutative quasigroup. Put
$$S^1=\{((x,i),(y,i),(z,i+1\ {\rm mod}\ 3)) \ |\ (x,y,z)\in A, i\in
\{0,1,2\}\},$$
$$ S^2=\{((x,0),(x,1),(x,2)) \ | \ x\in \{0,1,\dots,n-1\}\}.$$
Then $S^1\cup S^2$ is a STS of order $3n$.  Consider a latin square
$S_{3n}$ of order $3n$ corresponding to a STS obtained by Bose's
construction from a latin square $A$ of order $n$. Each transversal
$T$ of $A$ generates the transversal $T'$ of $S_{3n}$ by the
following rule. If $(a,b,c)\in T$ then
$((a,0),(b,0),(c,1)),((a,1),(b,1),(c,2)), ((a,2),(b,2),(c,0))\in
T'$, if  $(a,a,a)\in T$ then
$((a,0),(a,1),(a,2)),((a,1),(a,2),(a,0)), ((a,2),(a,0),(a,1))\in
T'$.

Let $A$ be a latin square of order $n$ with $k$ disjoint
transversals $T_0,\dots,T_{k-1}$. Construct the latin square
$\widehat{A}_k$ of order $n+k$ in the following way. If $(a,b,c)\in
A\setminus\cup_{i=0}^{k-1} T_i$ then $(a,b,c)\in \widehat{A}_k$. If
$(a,b,c)\in T_i$ then $(a,b,n+i),(a,n+i,c),(n+i,b,c)\in
\widehat{A}_k$. Moreover in the intersection of rows and columns
labeled by additional symbols $n,\dots,n+k-1$ we  substitute a latin
square $C$ of order $k$ on the alphabet $\{n,\dots,n+k-1\}$. For
example, if  $A=\left(
 \begin{array}{ccc}
0 & 1 & 2 \\
1 & 2 & 0 \\
2& 0 &  1\\
  \end{array} \right)$ and $T_0=\{(0,1,1),(1,2,0),(2,0,2)\}$,
   $T_1=\{(0,2,2),(1,0,1),(2,1,0)\}$ then
$\widehat{A}_2=\left(
 \begin{array}{ccc|cc}
0 & 3 & 4  &1 & 2\\
4 & 2 & 3 & 0 & 1\\
3& 4 &  1 & 2 & 0\\ \hline
2 & 1 & 0 & 3 & 4\\
1& 0 &  2 & 4 & 3\\
  \end{array} \right)$.
Let $t(A;T_0,\dots,T_{k-1})$ be the number of transversals of $A$
which don't intersect with  transversals $T_0,\dots,T_{k-1}$. It is
easy to see that $t(\widehat{A}_k)\geq t(C)t(A;T_0,\dots,T_{k-1})$.
If $k\neq 2$ then there exists $C$ with $t(C)\geq 1$.

It is easy to see that the quasigroup $q$ defined by the rule
$q(x,y)+q(x,y)\equiv x+y\ {\rm mod}\ n$ is idempotent, commutative
and isotopic to cyclic group of order $n$ as $n$ is odd. By
Proposition 1, the  latin square $B'_n$ that corresponds to
quasigroup $q$ has $n$ disjoint transversals. By means of Bose's
construction, we can obtain the Steiner latin square $S_{3n}$ from
$B'_n$. $S_{3n}$ has at least $n$ disjoint transversals. Thus we can
construct a latin square $D=\widehat{(S_{3n})}_k$ of order $3n+k$,
$k\leq n$, as described above. For $D$ we can repeat the reasonings
of the proof of Theorem \ref{t01transv}. The only difference is the
number of variants to choose  a triple at each step. The number of
variants for $D$ is $kn$ less  than for $S_{3n}$.

\begin{corol}
 $\frac{n}{6}(\ln n+ O(1)) \leq \ln T(n)\leq n(\ln n -2+o(1))$ as $n\rightarrow \infty$.
\end{corol}
The upper bound is provided by  (\ref{tAT}). The lower bound follows
from Theorem \ref{t01transv} and Stirling's formula if $n\equiv 1\
{\rm or}\ 3\ {\rm mod}\ 6$, and it is proved analogously in other
cases.

After submitting this paper to the journal, another paper \cite{GZ}
has appeared that contains a better  lower bound for the number of
transversals in latin squares. This lower bound is asymptotically
equal to the upper bound (\ref{tAT}).  However, the result in
\cite{GZ} is obtained using a non-constructive approach as oppose to
the method in our paper.

\section{Acknowledgments}
The author is grateful to Denis Krotov and Anna Taranenko for useful
discussions.

\end{document}